\documentclass[12pt]{amsart}
\usepackage{amsmath,amsthm,amsfonts,amscd,amssymb,eucal,latexsym}

\newcommand{\cl}[1]{{\mathcal{#1}}}

\numberwithin{equation}{section}

\theoremstyle{plain}
\newtheorem{lem}{Lemma}[section]
\newtheorem{thm}[lem]{Theorem}
\newtheorem{cor}[lem]{Corollary}
\newtheorem{pro}[lem]{Proposition}

\newtheorem{defn}[lem]{Definition}
\theoremstyle{definition}

\theoremstyle{remark}
\newtheorem{rem}[lem]{Remark}

\def\cl{{\mathcal L}}
\def\cam{{\mathcal M}}

\def\cas{{\mathcal S}}

\def\bc{{\mathbb C}}

\def\bn{{\mathbb N}}

\def\br{{\mathbb R}}

\def\bt{{\mathbb T}}
\def\bz{{\mathbb Z}}

\def\a{\alpha}
\def\eps{\varepsilon}

\def\b{\beta}
\def\d{\delta}

\numberwithin{equation}{section}

\def\cl{{\mathcal L}}
\def\cam{{\mathcal M}}

\def\cas{{\mathcal S}}

\def\bc{{\mathbb C}}

\def\bn{{\mathbb N}}

\def\br{{\mathbb R}}

\def\bt{{\mathbb T}}
\def\bz{{\mathbb Z}}

\def\a{\alpha}
\def\eps{\varepsilon}

\def\b{\beta}
\def\d{\delta}

\def\s{\sigma}

\begin{document} 

\title{Commutator Inequalities via Schur Products}
\author[E.~Christensen]{Erik Christensen}
\address{\hskip-\parindent
Erik Christensen, Institute for Matematicak Sciences, University of Co-penhagen, Denmark.}
\email{echris@math.ku.dk}
\date{\today}

\begin{abstract}
 For a self-adjoint unbounded operator $D$ on a Hilbert space $H,$
a bounded operator $y$ on $H$ and some  Borel functions $g(t)$ we establish inequalities of the type $$\|[g(D),y]\| \leq A_0\|y \| + A_1\|[D,y]\| + \dots + A_n\|[D, [D, \dots [D, y]\dots ]]\|.$$ 
The proofs take place in a space of infinite matrices with operator entries, and in this setting it is possible to approximate  the matrix associated to  $[g(D), y] $ by  the Schur product of a  matrix approximating $[D,y] $ and a scalar matrix.  A classical inequality on the norm of Schur products  may then be applied to obtain the results.
\end{abstract}

\maketitle

\section{Introduction}
\label{intro}
Let $D$ be an unbounded self-adjoint operator on a Hilbert space $H,$ then for a bounded operator $y$ on $H$ the  commutator $[D,y]$ may be of interest for several reasons. From our personal point of view, we have previously studied {\em mathematical physics, noncommutative geometry} and {\em operator algebras} and found that commutators do occur for natural reasons. The most common one being that the expression $[D,y] := Dy-yD$ denotes the derivative of the function $t \to -ie^{itD}ye^{-itD}$ at $t=0.$ It is well known \cite{AMG}, that this derivative may or may not exist, and if it exists, it may be as a limit of difference quotients taken in the norm topology or as a limit in the ultraweak toplogy on $B(H),$ and it follows from \cite{RP} and \cite{JaN}  that the ultraweak limit may exist in cases  where the uniform limit does not exist. If the ultraweak limit exists, we say that $y$ is weakly-$D$ differentiable and the bounded operator which is the limit is denoted $\d_D(y),$ 
 and if no mistakes are possible we will just write $\d(y).$ Following \cite{AMG} we let $C^1(D)$ denote the algebra of weakly $D-$differentiable operators on $H, $ and in the article   \cite{EC1} we gave a presentation of some of the equivalent formulations of weak $D-$differentiability. Our personal contribution in this direction 
 is based on a study of a space of infinite matrices with bounded operators as entries, and we showed that for any bounded operator $y$  the matrix associated with $[D,y]$ is always defined in this set up, but it represents a bounded operator exactly when $y$ is weakly $D-$differentiable. This set up makes it possible to define all the algebraic operations needed to work with higher commutators, although the higher commutators may not exist as densely defined operators. In the article \cite{EC2} we continued this work by studying higher commutators such as $[D,[D,\dots ,[D,y]\dots]]$ and we showed - among other things - 
that an operator $y$ is $k-$ times weakly $D-$differentiable if and only if all the matrix commutators $[D, [D, \dots[D,y]\dots ]]$ up to the order $k$ represent bounded operators.

Based on a question coming from {\em noncommutative geometry} and a question coming from the study of {\em perturbations of operator algebras } we got interested in relations between different commutators applied to the same bounded operator $y.$ To be more precise, the question from non commutative geometry asks for relations between the commutators $[|D|, y|$ and $[D,y],$ where we define for $D$ its numerical value $|D|$ by  $|D| := (D^2)^{(1/2)}.$ We show that for any natural number $n$ and a bounded operator $y$ on $H,$ that if  $y$ is $n+1$ times weakly $D-$differentiable then it is  $n$ times weakly $|D|-$differentiable, and we obtain a norm estimate for the norm of the $n'$th derivative of $y$ with respect to $|D|$ expressed in terms of norms of the $n+1$ first derivatives of $y$ with respect to $D.$  These results are presented in Section \ref{Holder}.

The question raised by the perturbation of operator algebras dealt with the question whether for a positive operator $D$, {\em almost  commutation }  with a bounded operator $y$ is inherited by the square root $D^{(1/2)}.$ It turned out that some results of this type are known, see for instance \cite{GKP}, so even though the set up we use in this article can be used to obtain some results of this type, we can not produce inequalities as sharp as the known ones, so our  results of this sort are not included in this article.

 It turns out that for some complex Borel functions $g(t)$ on $\br$  the relations between the  commutators $[g(D), y]$ and $[D,y]$ may,  approximately,  be expressed as a Schur product, and this is the basic observation, which is behind  the results in this article,  as described in Section \ref{basic}. The Schur  product between scalar matrices $A=(a_{ij}) $ and $B=(b_{ij})$  is simply the element-wise product  $(a_{ij})*(b_{ij}) = ( a_{ij}b_{ij} )$ of the   two matrices. With the kind help of Vern Paulsen we found a result by Bennett \cite{GB} on the norm of a Schur product of two scalar matrices, which we can use in our computations, and in Section \ref{ectSchur} we have modified Bennett's result from the setting of scalar matrices to matrices of operators and in this way we have found solutions to the problems mentioned above and also to some   general results of the type  $$\|[g(D),y]\| \leq A_0\|y\| +A_1\|[D,y]\| + \dots + A_n\|[D, \dots [D, y] \dots ]\|,$$ 
which hold, when  $g(t)$ is a Borel function such that 
$$ \exists \a> 0, \, A \geq 0, \, B \geq 0\, \forall s, t \in \br : \quad |g(s) - g(t)| \leq A + B |s-t|^\a.$$ This applies for instance to the cases when  $g(t) $ is bounded or $g(t)$ is H{\"o}lder continuous or $g(t)$ is absolutely continuous with a derivative which may be written as a sum of an integrable function and an essentially bounded function.  We present our abstract   results in  Section  \ref{Holder}, where we also give applications to the cases where $g(t) = \log(t)$ and $g(t) = |t|.$

\section{ Schur products, row- and column-bounded matrices.} \label{ectSchur}
This section contains a single result, which is an {\em operator space version}  of a result by Bennett, \cite{GB} Theorem 1.1 point (i). Given a couple of  scalar matrices $A = (a_{ij})_{(i,j) \in I \times J}$     and  $B= (b_{ij})_{(i,j) \in I \times J}$  over the same pair of sets of indices $I$ and  $J,$ then   the Schur product $C:= A*B$ is defined as  the matrix, of the same size, formed as the  products of the entries, $c_{ij} := a_{ij}b_{ij}.$  Bennett thought of  such   matrices over the scalars as operators between spaces $\ell^p(J)$ and  $\ell^q(I).$ In particular he proved that if the matrix $B$ corresponds to a bounded operator in $B(\ell^2(J), \ell^{\infty}(I))$ and the matrix $A$ corresponds to a bounded operator in $B(\ell^1(J), \ell^{2}(I)),$ then $A*B$ is a bounded operator in $B(\ell^2(J), \ell^2(I)).$ 
The theory of Schur multipliers has been extended to the more general setting of operator spaces and operator modules, see for instance Paulsen's book \cite{VP} and the books by Pisier \cite{GP1} and  \cite{GP2}.   A closer look at the definition of the Schur product - $c_{ij} = a_{ij}b_{ij} $ - immediately tells that the product  may be meaningful in  many different  situations, such as the case when all the elements $a_{ij}$ and $b_{ij} $ are  operators on the same Hilbert space, but also in the quite unrelated situation where the elements $a_{ij}$ are scalars and the  $b_{ij} $ are elements in a family $B_{ij}$ of Banach spaces.  We are sure that we have not found all the possible generalizations of Bennett's theorem to the setting of operator spaces, but we have searched and found  2 versions, which are   useful for this article and for our general interests.
 
The language we use is based on the words {\em column- and row-bounded } matrices of operators, which we define below.

The first version is formed as a statement on infinite matrices over the algebra of bounded operators algebras on a Hilbert space $H.$ This result  will not be used directly in this article, but we find quite easily  a corollary, which we will use  several times, in the sections to come.

\begin{defn}
Let $J$ be an index set and $H$ a Hilbert space, the set of all matrices $S= (s_{ij})_{(i,j \in J)}$ with entries   $s_{ij} \in B(H)$ is denoted $M_J(B(H)).$ 
\begin{itemize}
\item[(i)] 
 \end{itemize} 
A matrix $S= (s_{ij})$ in $M_J(B(H))$  is said to be bounded if it represents a bounded operator in $B(H \otimes \ell^2(J)).$ If $S$ is bounded, the term $\|S\| $ means that norm.  
\begin{itemize}
\item[(ii)] 
 \end{itemize} 
A matrix $S= (s_{ij})$ in $M_J(B(H))$  is said to be row bounded if there exists a $K \geq 0$ such that  $$ \forall i \in J \, \forall J_0 \text{ finite } \subseteq J: \, \| \sum_{j \in J_0} s_{ij}s^*_{ij} \| \leq K^2.$$
The row norm $\|S\|_r$ is defined as the least possible such $K.$  
\begin{itemize}
\item[(iii)] 
 \end{itemize} 
A matrix $S= (s_{ij})$ in $M_J(B(H))$  is said to be  column  bounded if there exists a $K \geq 0$ such that  $$ \forall j \in J \, \forall J_0 \text{ finite } \subseteq J: \, \| \sum_{i \in J_0} s_{ij}^*s_{ij} \| \leq K^2.$$ 
The column norm $\|S\|_c$ is defined as the least possible such $K.$  
\begin{itemize}
\item[(iv)] 
 \end{itemize} 
Let $K=(k_{ij}) $ and  $L= (l_{ij})$ be matrices in $M_J(B(H))$ their Schur product is the matrix $K*L$ in $M_J(B(H))$ defined by $(K*L)_{ij}:= k_{ij}l_{ij} .$
\end{defn} 

In this setting our version of  Bennett's result becomes 
\begin{thm} \label{schur}
Let $K=(k_{ij}) $ and  $L= (l_{ij})$ be matrices in $M_J(B(H))$ such that $K$ is row bounded and $L$ is column bounded, then $K*L $ is bounded in $M_J(B(H))$ and $\|K*L\| \leq \|K\|_r\|L\|_c.$  
\end{thm}
\begin{proof}
 Let $J_0 \subseteq J $ be a finite set and $ \xi =(\xi_j)_{(j \in J_0)} $ be in $\ell^2(J_0, H),$  then we will  estimate the sum of non negative terms $$ \sum_{i \in J}\|\sum_{j \in J_0} k_{ij}l_{ij} \xi_j\|^2$$ by first fixing $i$ and use the row boundedness of $K $ to see that
$$ \|\sum_{j \in J_0} k_{ij}l_{ij} \xi_j\|^2 \leq \|K\|_r^2 \sum_{j \in J_0}\|l_{ij}\xi_j\|^2. $$ Since we are summing non-negative reals, we may change the order of summation so we get from the last inequality, when we do the summation over $i$ first, the following inequality
\begin{align} \sum_{i \in J}\|\sum_{j \in J_0}k_{ij}l_{ij}\xi_j\|^2 &\leq \|K\|_r^2 \sum_{j \in J_0}\sum_{i \in J}\|l_{ij}\xi_j\|^2 \nonumber \\ &\leq \|K\|_r^2 \|L\|_c^2\sum_{j \in J_0}\|\xi_j\|^2 \nonumber \\ &= \|K\|_r^2 \|L\|_c^2\|\xi\|^2, \nonumber 
\end{align}
  and the theorem follows.  
\end{proof}
\begin{rem}
Matrices may be indexed by different indices for columns and rows, say $I$ and $J.$ The theorem above applies to this situation too, since such matrices embed naturally into {\em square matrices} over the set $I \cup J.$ \end{rem}

The proof above applies to a situation which looks  quite different. We will now look at a Hilbert space $ H$ and a family $(e_j)_{(j \in J)}$ of pairwise orthogonal projections in $B(H)$ with strong operator sum $I_{B(H)}$. Then we define $\cam$ as the linear space of all matrices $A = (a_{ij})_{(i,j \in J)} $ and $a_{ij} \in e_iB(H)e_j.$ The space $\cas$ is defined as all the square scalar matrices over the index set $J.$ A matrix $S= (s_{ij})$ in $\cas$ is said to be bounded if it is the matrix of a bounded operator on $\ell^2(J, \bc),$ and similarly a matrix $A = (a_{ij})$ in $\cam$ is bounded if there exists a bounded operator $y$ on $H,$ such that for each pair $i,j$ we have $a_{ij} = e_iye_j.$ The definitions of row- and column boundedness of $S$ and $A$ are made in the obvious way. Given a matrix $S =(s_{ij})$ in $\cas$ and a matrix $ A =(a_{ij}) $ in $\cam,$ their Schur product $S*A = ( s_{ij}a_{ij}) $ is a matrix  in $\cam$ and we get a new version of Bennett's result
\begin{cor} \label{schurcor}
Let $S=(s_{ij}) $ and  $A= (a_{ij})$ be matrices in $\cas$ and $\cam$ such that $S$ is row bounded and $A$ is column bounded, then $S*A$ is bounded and $\|S*A\| \leq \|S\|_r\|A\|_c.$  
\end{cor}
 \begin{proof}
 The proof is the same as that of Theorem \ref{schur}, even though the meaning of the product $s_{ij} a_{ij} $ is not identical to the meaning of the product $k_{ij}l_{ij} $ which appears in the previous proof. 
 
 It is also possible to apply the theorem directly and think of  both $S$  and $A$ as  elements in $M_J(B(H)).$   
 \end{proof}

\section{The basic example. } \label{basic}

The results we present in this article  are  based on some  well known aspects of elementary harmonic analysis. The unit circle $\bt$ is equipped with  the Hilbert space  $H: = L^2(\bt, (1/2\pi) d\theta) $  based on the Haar probability measure, and the self-adjoint operator $D := (1/i)\frac{d}{d\theta} $ is just differentiation with respect to arc length, multiplied by $-i.$ Basic Fourier analysis tells us that $H$ has an orthonormal basis $e_n := e^{in\theta}$ of eigenvectors for $D$ and with respect to this basis $D$ has a matrix representation as a diagonal matrix with all the integers as the  diagonal elements. Our basic observation is that if $y$ is a bounded operator on $H$ with matrix $(y_{ij})$ with respect to the basis $(e_n),$ then  the commutator $[D,y]$ makes sense as an infinite matrix even if there is no densely defined operator associated to this expression. Further the commutator may be described as a Schur product of the matrix $(y_{ij}) $ by the matrix $S = (S_{ij})$ given as $S_{ij}  = (i-j).$ If $g$ is a complex function on $\bz,$ then we see that if we define a matrix $$S:= (S_{ij}) \quad S_{ij} := \begin{cases} 0 \qquad \quad \text{ if } i =j \\
\frac{g(i) - g(j)}{i - j} \, \, \,  \text{ if } i \neq j \end{cases}, $$
then the commutator $[g(D), y] $ has a matrix which is given as the Schur product $(S_{ij})*([D,y]_{ij}).$
The next sections will show some examples on how this basic relation may be transformed, such that it works for a general self-adjoint operator $D,$ which may have a non-trivial continuous 
spectrum.

\section{The space $\cam_D$ of infinite matrices associated with $D$. } \label{cam}

Let $D$ be a self-adjoint operator on a Hilbert space $H$ and $y$ a bounded operator on $H,$ then we follow   \cite{AMG}, and say that $y$ is $n$ times weakly $D-$differentiable, if for any pair of vectors $\xi, \eta$ in $H$ the function $$\br \ni t \to \langle e^{itD}ye^{-itD}\xi, \eta \rangle $$ is $n$ times differentiable. 
In the articles \cite{EC1} and \cite{EC2} we presented several equivalent properties to being $n-$times weakly $D-$differentiable and pointed out that most of these may be found in the text book \cite{AMG}, however one point of view was original, and we will recover the set up for this property here in order to set the stage for the arguments to come. If $y$ is weakly $D-$differentiable, then we introduced the weak $D-$ derivative $\d(y)$ in the introduction, and we quote from \cite{EC2}, that if $y$ is  weakly $D-$differentiable, then  
 \begin{align} \forall \xi, \, \eta \, \in H: \quad \frac{d}{dt} \langle e^{itD}ye^{-itD}\xi, \eta \rangle &= \langle e^{itD}\d(y)e^{-itD}\xi, \eta \rangle \nonumber \\
\mathrm{dom}([D,y]) = \mathrm{dom}(D) \quad \text{and} \quad  \d(y) \big| \mathrm{dom}(D) &= i[D,y] .\nonumber
 \end{align}
 
 For an arbitrary bounded operator $y$ the expression $i[D,y]$ is an operator on $H,$ but it may have a small domain of definition, and it may not even be densely defined. In the basic example in Section \ref{basic} the commutator $i[D,y]$  always exists as an infinite matrix, and we will now present our way of transforming this {\em set up } to the case where the spectrum of $D$ has a continuous part.  
 
 Given the self-adjoint operator $D$  we define a sequence of spectral projections $e_n $ for $D,$ indexed  over the integers $\bz,$ by defining $e_n$ as the spectral projection for $D$ corresponding to the interval $[ n-(1/2), n+(1/2)[.$ Many of these projections may  vanish, but this will have no effect on the computations to come. Then we define $\cam_D$ as a space of matrices where the entries are bounded operators by 
 $$ \cam_D:= \{(x_{ij})\,: \, i, j \in \bz, \text{ and }  x_{ij} \in e_iB(H)e_j\, \}.$$ 
 
 Any bounded operator $y$ on $H$ has a natural representative $m(y)$ in $\cam_D$ which is given by $$m(y)_{ij} := e_iye_j.$$
For each integer $j$ the operator $d_j:= De_j$ is bounded, defined on all of $H$ and an element in $e_jB(H)e_j,$ so we can define an element $m(D) $ in $\cam_D$ by $$ m(D)_{ij} := \begin{cases} 0 \text{ if } i \neq j \\
 d_j \text{ if } i = j.\end{cases}$$
 
 For any element $x$ in $\cam_D$  the commutator $[m(D),x]$ is a well defined element in $\cam_D$ which we denote 
 $d(x) = \big( d(x)_{ij}\big)$ and it is given by
  $$ d(x)_{ij}  := [m(D),x]_{ij} := d_ix_{ij} - x_{ij}d_j.$$ 
 In this way, any iterated commutator such as  $$d^k(x) = [m(D), [ m(D), \dots [m(D), x]\dots ]]$$ is meaningful and we showed in \cite{EC2} that a bounded operator $y$ on $H$ is $n$ times weakly $D-$differentiable if and only if all the iterated commutators $d^k\big(m(y)\big)$ for $1 \leq k \leq n$ represent bounded operators on $H.$  

It is - of course - still just as difficult to prove boundedness for matrices  as to establish boundedness on a given domain of vectors, but the matrix language makes it possible to  assign a meaning to the expression $[D,y]$ in situations where the operator theoretical meaning is not existing. 
Further the mapping $d: \cam_D \to \cam_D$ is clearly linear even if there may exist bounded operators $y$ and $ z$ on $H$ such that   $[D, y+z] \neq [D,y] + [D, z].$   

There is still another advantage in working with $\cam_D,$ namely that $\cam_D$ is left invariant by Schur multiplication by scalar matrices $S= (S_{ij})_{(i,j \in \bz)}.$ We realized this phenomenon, when working with the unit circle $\bt$ as described in Section \ref{basic}. In that section we describe how the matrix of a commutator $[g(D),y]$ may be obtained as a Schur product of a scalar matrix $(S_{ij})$ and the matrix for the commutator $[D,y].$ 
It is not possible to lift this simple identity to the case where $D$ has continuous spectrum, since $$ [m(D), m(y)]_{ij} = d_iy_{ij} - y_{ij}d_j$$ and in general the expression $d_jy_{ij} $ is not defined, so we can not imitate the expression $(i-j)y_{ij}$ from the basic example. The general strategy in the sections to come is to replace $D$ by a perturbed operator $\overline{D}$ with pure point spectrum, such that $D-\overline{D}$ is bounded and satisfies $\|D-\overline{D}\| \leq 1/2.$  We define $\overline{D} $ as a Borel function of $D.$ 

\begin{equation}
\overline{D} := \mathrm{closure}\big(\sum_{n \in \bz}n e_n  \big), \label{Dbar}
\end{equation} 
and then  we may define a bounded operator $b$ as the infinite strong operator sum below, such that  
\begin{equation}
b:= \sum_{n \in \bz} (De_n - ne_n), \, \, \, \, \|b\| \leq \frac{1}{2}, \, \,  \, \, D= \overline{D} + b. \label{b} 
\end{equation}

It is now easy to check that for any bounded operator $y$ we have 
$$[m(\overline{D}),m(y)]_{ij} = (i-j)y_{ij},$$
so differentiation with respect to $\overline{D}$ may be expressed as a Schur product.  
\section{On $\|[g(D),y]\|$ for some  Borel  functions  $g(t).$} \label{Holder}

The usage of the Schur product, inside the space of matrices $\cam_D$ in the study of commutators of the type $[g(D), x],$ is easily demonstrated when $g(t) $ is H{\"o}lder continuous, i.e. there exist positive constants $\a$ and $C$ such that for any real numbers $s$ and $t$ we have $|g(s) - g(t)| \leq C|s-t|^\a.$ In this case we say that $g(t) $ is $(\a, C)$ H{\"o}lder continuous. However, it turns out that the proof of such a result does not need as strong a property as H\"older  continuity, instead we only need  a  property which we define below and denote H\"older boundedness. 

\begin{defn} \label{HoldBd}
Let $\a >0,\,  A \geq 0 , \, B \geq  0$ and $g(t) $ be a   complex Borel function on $\br,$ then $g(t)$ is said to be $(\a,A,B)$ H\"older bounded if $$ \forall \, s,t \in \br: \quad |g(s) - g(t) | \leq  A + B|s-t|^\a.$$
\end{defn}

We have searched the literature for places where this property has been used or studied, but so far without any success. 
 
It should be remarked, that the methods we propose also may fit to situations where $g(t) $ is only defined and H{\"o}lder bounded  on a subset of $\br$ containing the spectrum of $D$, but the details needed for a general theorem seems to be somewhat  complex, so we suggest that possible extensions in this directions are made on an {\em ad hoc } basis, if needed. 

\begin{thm}  \label{HoldThm}
Let $D$ be a self-adjoint operator on a
 Hilbert space and let $g(t) $ be an $(\a, A, B)$  H{\"o}lder bounded Borel  function on $\br.$ Let $n$ be the smallest integer such that $n > \a + 1/2$ and $y$ be a bounded operator which is $n-$times weakly $D-$ differentiable then $y$ is weakly  $g(D)-$differentiable and 
 
$$ \|[g(D), y]\| \leq 2(A+B) \bigg(\|y\| +  \sqrt{\frac{n -\a}{2n - 2\a-1}}\sum_{k=0}^n \binom{n}{k} \|\d^k(y)\| \bigg).$$
\end{thm}

\begin{proof}
We will use the {\em discrete } approximant $\overline{D}$ as introduced in (\ref{Dbar}), and we note from (\ref{b}) that the difference $\overline{D} -D $ is bounded, with closure $b$  and of norm at most $1/2.$ Then by the functional calculus for $D$  we find that the difference $ g(\overline{D}) - g(D)$ is densely defined and bounded of norm at most $A + B(1/2)^{\a} \leq A+B.$ This implies, that for any bounded operator $y$ we have \begin{equation}
\|[\big(m(g(D)- m(g(\overline{D}))\big), m(y)]\| \leq 2(A+B)\|y\|. \label{gD-gD)} \end{equation}
 
The Schur multiplier we need will be given as the matrix 
\begin{equation}
 S= (S_{ij})_{(i,j \in \bz)}: \quad S_{ij} = \begin{cases} 0 \qquad  \quad  \text{ if } i=j \\
\frac{g(i)  - g(j)}{(i -j)^n } \, \,  \text{ if }   i \neq j\end{cases}, \label{Sij} \end{equation} 
and then for any $(x_{ij}) $ in $\cam_D$ we can express the matrix $[m(g(\overline{D})), x]$ as the Schur product $S*[m(\overline{D}), [m(\overline{D}, \dots,[m(\overline{D}, x]\dots ]] ),$ with $n$ commutators. In order to be able to use Corollary  \ref{schurcor} we need to compute an estimate for the row norm of $ S$ and we get for any $i$ in $\bz$ that \begin{align} \label{gsum}
\sum_{j \in \bz}|S_{ij}|^2 &\leq 2 \sum_{k=1}^\infty \bigg(\frac{A + Bk^ {\a}}{k^n}\bigg)^2 \nonumber \\ & \leq 2(A+B)^2\sum_{k=1}^\infty k^{2\a- 2n} \nonumber \\ & \leq  2(A+B)^2(1 + \frac{1}{2n - 2\a-1}) \nonumber \\ &= 4(A+B)^2\frac{n -\a }{2n - 2\a-1}, 
\end{align}
so the row-norm is dominated as \begin{equation}\|S\|_r \leq 2(A+B)\sqrt{(n-\a)/(2n-2\a-1)}.\end{equation} 
 
 This inequality may be used to compare   expressions like $\|[g(\overline{D}), y]\|$ and  $\|[\overline{D}, [\overline{D}, \dots ,[\overline{D},y] \dots ]]\|,$   but we are interested in the terms based on $D$ instead of $\overline{D},$ so we will introduce the operators $\overline{d}  $ and $f$ on $\cam_D$ which are defined by 
 \begin{align}  \label{f} 
 \forall x \in \cam_D :&\quad \overline{d}(x) := m(\overline{D})x - xm(\overline{D} ) = [m(\overline{D}), x] \nonumber \\
 \forall x \in \cam_D :&\quad f(x) := m(b)x - xm(b) = [m(b), x], \text{ and} \nonumber \\
 \forall x \in \cam_D :&\quad x \text{ bounded } \implies f(x) \text{ bounded and } \|f(x) \| \leq \|x\| \nonumber \\
 \forall x \in \cam_D :&\quad d(x) \, = \overline{d}(x) + f(x). \nonumber
 \end{align}
The operators $d, f, \overline{d}$ on the linear space $\cam_D$ all commute, since both $\overline{D}$ and $b$ are functions of $D,$  and any left multiplication operator commutes with any right multiplication operator. The norm estimate above  follows since $\|b\| \leq 1/2.$ By elementary algebra we get that $$\overline{d}^n = \sum_{k=0}^n\binom{n}{k} (-f)^{n-k}d^k$$   
and if $y$ is $n-$times weakly  $D-$ differentiable then it is also $n-$times weakly $\overline{D}-$ differentiable and satisfies 
\begin{equation}
\|\overline{d}^n(m(y))\| \leq \sum_{k=0}^n \binom{n}{k} \|\d^k(y)\| \label{destimate} 
\end{equation}
 Now assume that $ n $ is the least integer such that $n > \a + 1/2, $ and $y$ is bounded and $n-$times weakly $D-$differentiable then, even if some terms are infinite, we may  by (\ref{gD-gD)}) and (\ref{destimate}) estimate as follows  
\begin{align} \label{finholder}
&\|[m(g(D)), m(y)]\|  \\&\leq  \|[\big(m(g(D)) -m(g(\overline{D}))\big), m(y) ]\|  + \|[m(g(\overline{D})), m(y) ]\|  \nonumber \\  
& \leq \|[m(g(\overline{D})), m(y) ]\|+ 2(A+B)\|y\| \nonumber \\
&=  \|(S_{ij}) * \overline{d}^n(m(y)) ]\| + 2(A+B)\|y\| \nonumber \\ 
& \leq 2(A+B)\sqrt{(n-\a)/(2n - 2\a -1)}  \|\overline{d}^n(m(y)) \| + 2(A+B)\|y\| \nonumber \\
& \leq 2(A+B)\bigg(\|y\| +   \sqrt{(n-\a)/(2n - 2\a -1)}\sum_{k=0}^n \binom{n}{k}   \|\d^k(y)\| \bigg). \nonumber   
\end{align}

Hence we see that $[m(g(D)), m(y)] $ is bounded in $\cam_D,$   so $y $ is weakly $g(D)-$differentiable and the theorem follows.
\end{proof}

\subsection{On $\|[g(D),y]\|$ for some absolutely continuous functions.} \label{g}

In this subsection  $g(t)$ is assumed to be a complex valued absolutely   continuous function on $\br$  which has  a certain type of bound on its growth.  To be more precise we assume that $g(t)$ has a locally integrable derivative $g^\prime(t)$ which may be written as a sum of an essentially bounded function $u(t)$ and an integrable function $\ell(t).$
There exists a theory of interpolation in Banach spaces, \cite{LT}, and it is well known that it is possible to define a Banach space $L^1(\br) + L^\infty(\br)$ consisting of equivalence classes of sums $\ell(t) +u(t)$ with $\ell(t)$ integrable and $u(t)$ essentially bounded. The norm of an element $h(t)$ in the Banach space $L^1(\br) + L^\infty(\br)$ is given by 
\begin{align*}&\|h\|_{(L^1 + L^\infty)} := \nonumber \\ & \inf \{\|\ell\|_1 + \|u\|_\infty\, : \ell \in \cl^1(\br), u \in \cl^\infty(\br) \text{ and } h = \ell + u,  \text{ a.e.} \}.\end{align*} 

Before we formulate our theorem we will like to remind you that for any $ p \geq 1$ we have $L^p \subseteq L^1 + L^\infty$ such that for $f$ in $\cl^p$ we have $\|f\|_{(L^1+L^\infty)} \leq \|f\|_p^p + 1. $ This is an easy consequence  of the fact that if the set $U$ is defined as $U := \{ t \in \br : |f(t)| \leq 1\}$ then $ f$ may be written as $f*(1 - 1_U) + f*1_U $ which is a decomposition with the desired properties.    
\begin{thm} \label{g(D)}
Let $D$ be a self-adjoint operator on a Hilbert space $H$ and $y$ a bounded operator on $H.$ Let $g(t)$ be a complex absolutely  continuous function on $\br.$ If $g(t)$  has a derivative $g^\prime(t) $ in $\cl^1(\br) + \cl^\infty(\br)$  and  $y$ is in $C^2(D)$  then $y$  is weakly $g(D)-$differentiable and 
$$ \|[g(D),y]\| \leq  \|g^\prime\|_{(L^1+L^\infty)}\big(4\|y\| + 4\|\d(y)\| +  2\|\d^2(y)\| \big)  
$$
\end{thm}

\begin{proof}
We start by showing that $g(t)$ is H\"older bounded, so let $\eps >0$  and let  $\ell(t)$ in $\cl^1(\br)$ and $u(t)$ in $\cl^\infty(\br)$  be such  that $g^\prime(t) = \ell(t) + u(t), \text{ a.e.}$ and $ \|\ell\|_1 + \|u\|_\infty \leq  \|g^\prime\|_{(L^1+L^\infty)}+ \eps.$ Then for any  $s, t$ in $\br$ we have 

\begin{align}
|g(s) - g(t)|  &= |\int_t^s g^\prime(v) dv |  \nonumber \\ &\leq |\int_t^s \ell(v) dv|  + |\int_t^s u(v)dv| \nonumber \\
& \leq \|\ell\|_1 + (\|u\|_\infty) |s-t|,  \label{gineq}
\end{align}

so $g(t)$ is H\"older bounded with the constants $(1, \|\ell\|_1 , \|u\|_\infty)$ We can then apply Theorem \ref{HoldThm} to obtain
\begin{align}
\|[g(D),y]\| &\leq  2(\|\ell\|_1 + \|u\|_\infty) \big(2\|y\| + 2\|\d(y)\| +  \|\d^2(y)\| \big) \nonumber \\  
 &\leq (\|g^\prime\|_{(L^1+L^\infty)} + \eps )\big(4\|y\| + 4\|\d(y)\| +  2\|\d^2(y)\| \big),  
\end{align}
and the theorem follows. 
\end{proof}

We may obtain a result which only depends on $\d(y)$ if we assume that  $g^\prime(t)$ is in $\cl^p$ for a   $p$ in the interval $[1, 2[.$ 

\begin{pro} \label{Lp}
Let $D$ be a self-adjoint operator on a Hilbert space H, $y$ a bounded operator on $H,$  $1\leq p<2$ and   $g(t)$ an absolutely continuous function on $\br$ such that $g^\prime(t)$  is in $L^p(\br).$ 
If $y$ is weakly $D-$differentiable, then $y$ is weakly $g(D)$ differentiable and 
$\|[g(D),y]\|  \leq  2\|g^\prime\|_p\bigg( \big(1+1/\sqrt{2-p}\big) \|y\|+ \big(1/\sqrt{2-p}\big)\|\d(y\|\bigg).  $
\end{pro}

\begin{proof}
The cases $p=1$ and $1< p < 2$ have different proofs, and we will first assume that $p=1.$ Then for any pair $s,t$ of real numbers $$|g(s) - g(t)| = |\int_t^s g^\prime(v) dv| \leq \|g^\prime\|_1$$ so $g(t)$ is H\"older bounded for any $\a > 0$ with constants $(\a, \|g^\prime\|_1, 0).$ Then By Theorem \ref{HoldThm} we have for $0 < \a < 1/2$ that 
$$ \|[g(D), y]\| \leq 2 \|g^\prime\|_1\big(\|y\|+ \sqrt{\frac{1 - \a}{1- 2\a}}(\|y\| + \|\d(y)\|)\big),$$ and if we let $\a$ decrease to $0,$ the result follows in the case $p=1.$

 When $1 < p < 2$  H\"older's inequality gives 
$$|g(t)-g(s)| = |\int_s^tg^\prime(v)dv| \leq \|g^\prime\|_p |t-s|^{(p-1)/p},$$  so $g(t)$ is H\"older continuous with constants $((p-1)/p, \|g^\prime\|_p).$
Hence Theorem \ref{HoldThm} shows that 
we may use $n =1$ to obtain that if $y$ is weakly $D-$differentiable then it is weakly $g(D)$ differentiable with  
\begin{align} \|[g(D),y]\| &\leq 2\|g^\prime\|_p\big(\|y\| + \frac{1}{\sqrt{2-p}}(\|y\| + \|\d(y)\| )\big) \nonumber \\
& = 2\|g^\prime\|_p\bigg(\big(1+ \frac{1}{\sqrt{2-p}}\big)\|y\| + \frac{1}{\sqrt{2-p}} \|\d(y)\| \bigg)   \end{align}
and the proposition follows.    \end{proof}

\subsection{Relations between  commutators $[D,y]$ and commutators   $[\log(D),y].$ } \label{sec6}
 In this subsection we will show how Proposition \ref{Lp} may be applied to the situation where we study $\log(D)$ for   a positive, possibly unbounded  operator $D$ on a Hilbert space $H$ such that $D$ is either invertible or the point $\{0\}$ is an isolated point in the spectrum $\s(D).$ In either case there is a real number $\b$ which is the smallest positive value in $\s(D),$ and this value will play an important role in the estimates we are giving below.

For any $\b > 0$ we  then define an absolutely  continuous real function $g_\b(t)$ on $\br$ which will make the use of Proposition  \ref{Lp} possible. 

We define 
$$ \forall \b >0: \qquad g_\b(t): = \begin{cases} \log(t) \, \, \text{ if } t \geq \b \\
 \log(\b) \, \text{ if } -\infty <t <  \b \end{cases}, $$  
and then we see that the conditions for $g_\b(t)$ as given in Proposition  \ref{Lp} are satisfied for $p=3/2$ and that 
\begin{equation}
 \|g^\prime_\b\|_{(3/2)}  = 2^{(2/3)} \bigg(\frac{1}{\b}\bigg)^{(1/3)}. \label{G}
\end{equation}

 As an immediate consequence of Proposition \ref{Lp} we  get the following proposition 
 \begin{pro} \label{gbD}
Let $D$ be a positive self-adjoint operator on a Hilbert space $H,$ such that there exists a smallest positive value $\b$ in the spectrum, and let $y$ be  a bounded operator on $H.$ 
If $y$ is weakly $D-$differenti-able then it is weakly $g_\b(D)$ differentiable  
and 
$$ \|[g_\b(D),y] \| \leq \bigg(\frac{1}{\b}\bigg)^{(1/3)}\bigg(8\|y\| + 5\|\d(y)\|\bigg).$$
 \end{pro} 

We will also introduce the function $\widetilde{\log}(t)$ which is defined by 
$$ \widetilde{\log}(t) := \begin{cases} \log(t) \text{ if } t > 0 \\ 0 \text{ if } t \leq 0, \end{cases}$$
 and we find that if $D$ is positive and invertible  then $g_\b(D) = \widetilde{\log}(D).$ In order to investigate commutators with  $\widetilde{\log}(D)$ when $D$ is positive and not invertible, we let $E_0$ denote the spectral projection onto the kernel of $D,$  and then we find that $D + \b E_0$ is invertible plus
 
\begin{equation} \widetilde{\log}(D) = g_\b(D) - \log(\b) E_0, \label{E0} \end{equation} 
 so based on this we see that a bounded operator $ y $ is weakly  $\widetilde{\log}(D)-$dif-ferentiable   if and only if it is  weakly $g_\b(D)-$differentiable, and then 
   we can estimate the norm $\|[\widetilde{\log}(D), y]\|,$ so we obtain

 \begin{thm} \label{logNinv}
 Let $D$ be a positive  operator  on a Hilbert space such that there exists a minimal positive value $\b$  in $\s(D)$ and $y$ a bounded  operator on $H.$ If $y$ is weakly $D-$differentiable, then it is weakly $\widetilde{\log}(D)-$dif-ferentiable and 
 \begin{itemize}
 \item[(i)]$\, \, $ If $D$ is invertible then
 $$ \|[\widetilde{\log}(D),y] \| \leq \bigg(\frac{1}{\b}\bigg)^{(1/3)}\bigg(8\|y\| + 5\|\d(y)\|\bigg).$$ 
\item[(ii)] $\, $ If $D$ is not invertible then 
  $$ \|[\widetilde{\log}(D),y]\| \leq   \big(8(1/\b)^{(1/3)}+ |\log(\b)|\big)\|y\|  + 5(1/\b)^{(1/3)}\|[D,y]\| .$$ \,  \end{itemize} 
 \end{thm} 

\begin{proof}
The case (i) follows directly from Proposition \ref{gbD} since, when $D$ is invertible, $g_\b(D) = \widetilde{\log}(D).$ 
 To settle (ii). let us recall,  that for any bounded operator $y$ we have $$ [E_0,y] = E_0y(I-E_0) - (I-E_0)yE_0,$$ so $\|[E_0, y]\| \leq \|y\|,$ hence a combination of  the  equation (\ref{E0}) and Proposition \ref{gbD} implies the second estimate, so the theorem follows.
\end{proof}

It may be  worth to notice that when  $D$ is invertible, then  we get for any positive real $s,$ that $\widetilde{\log}(sD) = \log(s)I + \widetilde{\log}(D),$ so for $s > 0$ we have 
\begin{align}\|[\widetilde{\log}(D), y ]\| &=  \|[\widetilde{\log}(sD), y ]\|
& \leq \bigg(\frac{1}{s\b}\bigg)^{(1/3)}\bigg(8\|y\| + s(5\|\d(y)\|) \bigg). \label{infs}
\end{align} 

The minimum of the right hand side of (\ref{infs}) over $ s> 0$ is  $\b^{-(1/3)}*12*(5/4)^{(1/3)}*\|y\|^{(2/3)}\|\d(y)\|)^{(1/3)},$ so we have proven the following result.

\begin{cor}
If $D$ is invertible then 
$$\|[\widetilde{\log}(D),y]\|  \leq 13 \b^{-(1/3)} \|y\|^{(2/3)}\|\d(y)\|^{(1/3)}.$$
\end{cor}

Since $g^\prime_\b(t)$ is in $L^p(\br)$ for any $p > 1$ we could have chosen to optimize over $p > 1$ too, but we have decided to keep this article at a reasonable length.  

\subsection{ Norms of commutators   $[|D|, y].$ } \label{[|D|,y]}
The Theorem \ref{HoldThm} may be applied to the function $|t| $ which is $(1,1)$ H\"older continuous and then it is $(1,0,1)$ H\"older bounded, so we get right away the following result 

\begin{pro} \label{|D|}
Let $y $ be in $C^2(D)$ then $y $ is in $C^1(|D|)$ and $$
\|\d_{|D|}(y)\| \leq 4\|y\| + 4\|\d_D(y)\| + 2\|\d_{D}^2(y)\|.$$
\end{pro}

This result shows by induction, that for any natural number $n$ there must exist positive reals $A_0, \dots , A_{2n} $ such that if $y $ is in $C^{2n}(D)$ then $y$ is in $C^n(|D|) $ and 
$$\|\d_{|D|}^n(y)\| \leq A_0\|y\| + A_1\|\d_D(y)\| + \dots + A_{2n}\|\d_D^{2n}(y)\|.$$

However, refinements of the methods used in the proof of Theorem  \ref{HoldThm},
show that we can do much better. We have obtained several inequalities, which are not easily directly comparable, by changing the grid length in the construction of $\overline{D}.$  If the grid length is kept as $1,
$ then a direct application of the methods from above shows that we can obtain   
 \begin{thm}  \label{k+1}
Let $D$ be an unbounded operator on a Hilbert space $H,$  $y$ a bounded operator on $H$ and $n$ a natural number. If $y$ is $n+1$ times weakly $D-$differentiable, then $y$ is $n$ times weakly $|D|-$differentiable and 

\begin{equation}
 \|\d_{|D|}^n(y) \|   
\leq  2^n \frac{\pi}{\sqrt{3}}\|y\| +  \frac{\pi}{\sqrt{3}}\sum_{l=1}^{n+1}\binom{n+1}{l}2^{(n+1-l)}\|\d_D^l(y)\|. \nonumber
\end{equation}   
\end{thm}

\begin{proof}
The arguments are similar to the ones used in the proof of  Theorem \ref{HoldThm}, so we will leave out a few  details. 
For a diagonal element $z$ in $\cam_D $ we will let $d_z$ denote the linear mapping on $\cam_D$ given by $d_z(x)_{ij} = z_{ii}x_{ij }- x_{ij}z_{jj}$. The operators $z$ we will use, will  all be of the form $z = m(f(D))$ for some Borel function $f(t),$ so they will all commute as operators on $\cam_D.$  We will recall the definitions of $\overline{D}$ and $b$ from  (\ref{Dbar}) and (\ref{b}), but we will  add the definition of $c$ as the bounded operator which satisfies $|D| = |\overline{D}| + c.$ Then $c$ satisfies  $\|c\| \leq 1/2.$ For each natural number $k$ we let $S(k):= (S(k)_{ij}) $ denote the matrix in $\cas$ which is given by $$S(k)_{ij} = \begin{cases} 0 \qquad \quad \text{ if } i =j \\ \frac{(|i|-|j|)^k}{(i-j)^{k+1}} \,\,  \, \,\text{ if } i \neq j \end{cases}. $$ 
Since $||i|-|j|| \leq |i-j|,$ the square of the row norm equals $2\sum_{\bn}j^{-2} = \pi^2/3.$ 
We will let $F_k$ denote the operator on $\cam_D$ which consists in Schur multiplication by $S(k),$ then we have the following sequence of identities 
\begin{align} 
&d_{m(|D|)}^n = (d_{m(|\overline{D}|)} + d_{m(c)})^n \nonumber \\    
= & d_{m(c)}^n + \sum_{k=1}^n \binom{n}{k}  d_{m(c)}^{n-k} d_{m(|\overline{D}|)}^k \nonumber \\
=& d_{m(c)}^n + \sum_{k=1}^n \binom{n}{k}  d_{m(c)}^{n-k} F_k d_{m(\overline{D})}^{k+1} \nonumber \\
=&  d_{m(c)}^n + \sum_{k=1}^n \binom{n}{k}  d_{m(c)}^{n-k} F_k\big( d_{m(D)} -d_{m(b)}\big)^{k+1} \nonumber \\
=&  d_{m(c)}^n + \sum_{k=1}^n \binom{n}{k}  d_{m(c)}^{n-k} F_k( -d_{m(b)})^{k+1}\nonumber \\+& \sum_{k=1}^n \sum_{l=1}^{k+1}\binom{n}{k}\binom{k+1}{l}  d_{m(c)}^{n-k} F_k(-d_{m(b)})^{(k+1-l)} d_{m(D)}^l   \nonumber \\
\end{align}

Since for a bounded operator $y$ we have $\|cy-yc\| \leq \|y\|,$ $\|by-yb\| \leq \|y\| $ and $\|F_k(m(y))\| \leq \pi/\sqrt{3}\|y\|$ we may start computing. So let $y$ be $n+1$ times weakly $D-$differentiable then the identities above show 
\begin{align}
& \|d^n_{m(|D|)}(y) \|  \leq 2^n\pi/\sqrt{3} \|y\| +  \pi/\sqrt{3}\sum_{k=1}^n \sum_{l=1}^{k+1}\binom{n}{k}\binom{k+1}{l}  \|\d^l(y)\| \nonumber \\
&= 2^n\pi/\sqrt{3} \|y\| +  \pi/\sqrt{3}\sum_{l=1}^{n+1} \sum_{k=l-1}^{n}\binom{n}{k}\binom{k+1}{l}  \|\d^l(y)\| \nonumber \\ 
&\leq  2^n\pi/\sqrt{3} \|y\| +  \pi/\sqrt{3}\sum_{l=1}^{n+1} \sum_{u=0}^{n+1-l}\binom{n+1}{l}\binom{n+1-l}{u}  \|\d^l(y)\|
 \nonumber \\ 
 & = 2^n\pi/\sqrt{3} \|y\| +  \pi/\sqrt{3}\sum_{l=1}^{n+1} \binom{n+1}{l} 2^{(n+1-l)}  \|\d^l(y)\|
 \nonumber \\ 
\end{align}  
and the theorem follows. 
\end{proof}

If the grid length is made smaller, say of size 1/2, then the powers of $2$ will disappear from the theorem above, but the row norm of $S(k)$ will double. We have left these computations out, partly because they are rather long, and partly because we do not know if there exists an optimal grid length.

\end{document}